\documentclass[12pt]{article}
\usepackage[OT1]{fontenc}
\usepackage[latin1]{inputenc}
\usepackage{graphicx}
\usepackage{geometry}
\usepackage{hyperref}
\usepackage{amsmath}
 \usepackage{amssymb}
\usepackage{amsthm}
\usepackage[english]{babel}
\let\dps\displaystyle
\newtheorem{theorem}{Theorem}
\newtheorem{proposition}{Proposition}

\newtheorem{definition}{Definition}

\title{Groups of order 8 and 16}
\date{}
\author{J. Lapuyade-Lahorgue}
\begin{document}
\maketitle
\section{Introduction}
This document is inspirated of the work of David Clausen (University of Puget Sound, USA) with simplification in the proofs. The reader has to have basic knowledge on group theory including:
\begin{itemize}
\item Abelian and cyclic groups.
\item Lagrange theorem.
\item Normal subgroup and quotient group.
\item Direct and semi-direct products.
\item Cauchy theorem.
\item Operation of groups on a set.
\end{itemize}
\section{Preliminaries}
\begin{definition}[Center of a group]
Let $G$ be a group, the center of $G$ is defined as:
\begin{equation}
Z(G)=\left\{h:\forall g,hg=gh\right\}\nonumber
\end{equation}
\end{definition} 
\begin{definition}[Commutator subgroup]
Let $G$ be a group, the commutator of $G$ is the smallest subgroup containing the commutators $[g_{1},g_{2}]=g_{1}g_{2}g_{1}^{-1}g_{2}^{-1}$
\end{definition}
\begin{proposition}
If the quotient group $G/Z(G)$ is cyclic then $G$ is Abelian and consequently $G/Z(G)\cong\left\{e\right\}$.
\end{proposition}
\begin{proof}
Suppose $G/Z(G)=\left\{Z(G),gZ(G),g^{2}Z(G),\ldots,g^{n-1}Z(G)\right\}$. Let $g^{i}x_{i}$ and $g^{j}x_{j}$, with $x_{i},x_{j}\in Z(G)$ be two elements of $G$. Then:
\begin{eqnarray}
g^{i}x_{i}g^{j}x_{j}&=&g^{i}g^{j}x_{i}x_{j}\nonumber\\
&=&x_{j}g^{i+j}x_{i}\nonumber\\
&=&x_{j}g^{j}g^{i}x_{i}\nonumber\\
&=&g^{j}x_{j}g^{i}x_{i}\nonumber
\end{eqnarray}
\end{proof}
\begin{proposition}
If $|G|=p^n$, where $p$ is a prime number, then $|Z(G)|=p^{k}$ for $k\geq 1$.
\end{proposition}
\begin{proof}
$G$ operates on itself by conjugaison $g.h=ghg^{-1}$. For any $h$, the stabilizer of $h$ is the centralizer $S_{h}=\left\{g:gh=hg\right\}$ and its orbital is the set $\mathcal{O}(h)=\left\{ghg^{-1}:g\in G\right\}$. We recall that $\mathcal{O}(h)$ is in bijection with the quotient set $\left(G/S_{h}\right)_{l}=\left\{gS_{h}:g\in G\right\}$ and consequently has the same cardinal. It is easy to show that the cardinal $|\mathcal{O}(h)|=1$ if and only if $h\in Z(G)$, consequently:
\begin{equation}
|G|=|Z(G)|+\sum_{h\notin Z(G)}\frac{|G|}{|S_{h}|},\nonumber
\end{equation}
where the sum is indexed for the $h$ which are not in the same orbital.\\
The term $\dps{\sum_{h\notin Z(G)}\frac{|G|}{|S_{h}|}}$ is a multiple of $p$, consequently $|Z(G)|$ also. Using the Lagrange's theorem, we deduce the result.
\end{proof}
\begin{theorem}[Burnside's theorem]
If the order of a group $G$ is equal to $p^2$, with $p$ a prime number; then $G$ is abelian.
\end{theorem}
\begin{proof}
Easy: use the two previous propositions.
\end{proof}
\begin{proposition}
The commutator subgroup, denoted $D(G)$, is the smallest normal subgroup such that $\dps{G/D(G)}$ is Abelian.
\end{proposition}
\begin{proof}
\textbf{First step:} $D(G)$ is a normal subgroup:
\begin{eqnarray}
h[g_{1},g_{2}]h^{-1}&=&hg_{1}g_{2}g_{1}^{-1}g_{2}^{-1}h^{-1}\nonumber\\
&=&hg_{1}g_{2}g_{1}^{-1}(hg_{2})^{-1}\nonumber\\
&=&hg_{1}h^{-1}(hg_{2})g_{1}^{-1}(hg_{2})^{-1}\nonumber\\
&=&hg_{1}h^{-1}g_{1}^{-1}g_{1}(hg_{2})g_{1}^{-1}(hg_{2})^{-1}\nonumber\\
&=&[h,g_{1}][g_{1},hg_{2}]\in D(G)\nonumber
\end{eqnarray}
\textbf{Second step:} $\dps{G/D(G)}$ is Abelian:\\
It is trivial by construction of $D(G)$. Indeed, $g_{1}^{-1}g_{2}^{-1}g_{1}g_{2}\in D(G)$, consequently, $g_{1}g_{2}D(G)=g_{2}g_{1}D(G)$.
\textbf{Third step:} If $H$ is a normal subgroup such that $\dps{G/H}$ is Abelian, then $H$ contains the commutator. The conclusion is trivial.
\end{proof}
\begin{theorem}[Correspondance theorem]
Let $H$ be a normal subgroup of a group $G$, then there is a bijection between the set of the subgroups $S$ of $G$ containing $H$ and the set of the subgroups $S/H$ of $G/H$.
\end{theorem}
\begin{proposition}
Let $H$ and $K$ be two subgroups of a group $G$. We define the set $HK=\left\{hk:h\in H,k\in K\right\}$, which is not necessarily a group. Then:
\begin{equation}
|HK|=\frac{|H|\times |K|}{|H\cap K|}\nonumber
\end{equation}
\end{proposition}
\section{Groups of order 8}
We now classify all groups of order 8. The neutral element will be denoted $e$.
\subsection{Abelian groups of order 8}
\begin{center}
\begin{tabular}{|c|c|}
\hline
\textbf{Name} & \textbf{Character presentation}\\
\hline
$\mathbb{Z}_{8}$ & $a^{8}=e$\\
\hline
$\mathbb{Z}_{4}\times \mathbb{Z}_{2}$ & $a^{4}=b^{2}=e$ and $D(G)=\left\{e\right\}$\\
\hline
$(\mathbb{Z}_{2})^{3}$ & $a^{2}=b^{2}=c^{2}=e$ and $D(G)=\left\{e\right\}$.\\
\hline
\end{tabular}
\end{center}
\subsection{No Abelian groups of order 8}
\begin{proposition}
We have necessarily:
\begin{equation}
Z(G)\cong\mathbb{Z}_{2}\nonumber
\end{equation}
and:
\begin{equation}
G/Z(G)\cong\mathbb{Z}_{2}\times \mathbb{Z}_{2}\nonumber
\end{equation}
\end{proposition}
\begin{proof}
As $G/Z(G)$ cannot be cyclic and $|Z(G)|$ is a non trivial power of $2$ different of $8$, we deduce the result.
\end{proof}
From the previous proposition, as $G$ is not Abelian and $G/Z(G)$ is Abelian; we deduce that $D(G)=Z(G)$.\\
Moreover, we remark that:
\begin{equation}
\mathbb{Z}_{2}\times \mathbb{Z}_{2}=\left\{(0,0),(1,0),(0,1),(1,1)\right\},\nonumber
\end{equation}
has three subgroups of order $2$. Consequently, by the correspondance theorem, $G$ has three subgroups of order $4$, denoted $G_{1}$, $G_{2}$, $G_{3}$, containing the center.\\
\begin{proposition}
If $g_{1}\in G_{1}\backslash Z(G)$ and $g_{2}\in G_{2}\backslash Z(G)$, then $g_{1}g_{2}\in G_{3}\backslash Z(G)$.
\end{proposition}
\begin{proof}
Suppose, by the correspondance theorem, that $G_{1}$ (resp. $G_{2}$) corresponds to the subgroup $\left\{(0,0),(1,0)\right\}$ (resp. $\left\{(0,0),(0,1)\right\}$). Then, $g_{1}$ corresponds to $(1,0)$ and $g_{2}$ to $(0,1)$. As $(1,0)+(0,1)=(1,1)$ which corresponds to $g_{1}g_{2}$; we deduce the conclusion. 
\end{proof}
\begin{proposition}
If $g_{1}\in G_{1}\backslash Z(G)$ and $g_{2}\in G_{2}\backslash Z(G)$, then $g_{1}g_{2}\neq g_{2}g_{1}$.
\end{proposition}
\begin{proof}
Suppose the converse. Then $g_{1}$ commutes with the elements of $Z(G)$, of $g_{1}Z(G)$ and of $g_{2}Z(G)$. Consequently, the centralizer of $g_{1}$ has at least $6$ elements. And, as the centralizer is a subgroup of $G$ then it has exactly $8$ elements. We deduce that the centralizer of $g_{1}$ is $G$, so $g_{1}\in Z(G)$; which is contradictory.
\end{proof}
\begin{proposition}
Let $G_{i},G_{j}$ be two distinct subgroups amongst $G_{1},G_{2},G_{3}$, then:
\begin{eqnarray}
G&=&G_{i}G_{j}\nonumber\\
G_{i}\cap G_{j}&=&Z(G)\nonumber
\end{eqnarray}
\end{proposition}
\begin{proof}
Use the formula:
\begin{equation}
|G_{i}G_{j}|=\frac{|G_{i}|\times |G_{j}|}{|G_{i}\cap G_{j}|}=\frac{16}{|G_{i}\cap G_{j}|},\nonumber
\end{equation}
and the fact that $G_{i}G_{j}$ cannot have more elements than $G$.
\end{proof}
As the sub-groups $G_{i}$ has order equal to $4$, they are either isomorphic to $\mathbb{Z}_{4}$ or $\mathbb{Z}_{2}\times\mathbb{Z}_{2}$.

In the following, let $Z(G)=\left\{e,\alpha\right\}$ with $\alpha^2=e$.
\subsubsection{$G_{1},G_{2},G_{3}\cong\mathbb{Z}_{2}\times\mathbb{Z}_{2}$}
In this case, there are three elements $a,b,c$ in $G\backslash Z(G)$ such that:
\begin{itemize}
\item $G_{1}=\left\{e,a,\alpha,a\alpha\right\}$.
\item $G_{2}=\left\{e,b,\alpha,b\alpha\right\}$.
\item $G_{3}=\left\{e,c,\alpha,c\alpha\right\}$.
\end{itemize}
From the Proposition 6., $ba\in G_{3}$ and consequently has order $2$; so $baba=e$. But, as $a$ and $b$ have order $2$, then $baab=e$. Consequently, $ba=ab$, which contradicts Proposition 7.
\subsubsection{$G_{1}\cong\mathbb{Z}_{4},G_{2},G_{3}\cong\mathbb{Z}_{2}\times\mathbb{Z}_{2}$}
There exists three elements $a,b,c$ in $G\backslash Z(G)$ such that:
\begin{itemize}
\item $G_{1}=\left\{e,a,a^2=\alpha,a^3\right\}$.
\item $G_{2}=\left\{e,b,\alpha,b\alpha\right\}$.
\item $G_{3}=\left\{e,c,\alpha,c\alpha\right\}$.
\end{itemize}
$ba$ has order $2$, so $baba=e$. $a$ has order $4$ and $b$ has order $2$, so $baa^3b=e$. Consequently, $ba=a^3b$. As $G=G_{1}G_{2}$, $G$ is generated by $a$ and $b$. Moreover, $|<a>|\times |<b>|=4\times 2=8=|G|$ and $<a>\cap <b>=\left\{e\right\}$; $G$ is the semi-direct product $<a>\rtimes <b>\cong \mathbb{Z}_{4}\rtimes \mathbb{Z}_{2}$. It is the diedral group $D_{4}$. Its character representation is $a^{4}=b^{2}=e$ and $ba=a^{3}b$.
\subsubsection{$G_{1},G_{2}\cong\mathbb{Z}_{4},G_{3}\cong\mathbb{Z}_{2}\times\mathbb{Z}_{2}$}
There exists three elements $a,b,c$ in $G\backslash Z(G)$ such that:
\begin{itemize}
\item $G_{1}=\left\{e,a,a^2=\alpha,a^3\right\}$.
\item $G_{2}=\left\{e,b,b^2=\alpha,b^3\right\}$.
\item $G_{3}=\left\{e,c,\alpha,c\alpha\right\}$.
\end{itemize}
We have also $baba=e$ and $baa^3b^3=e$. Consequently $ba=a^3b^3=a^2abb^2=\alpha ab\alpha=\alpha^2 ab=ab$ because $\alpha\in Z(G)$ and has order $2$. $ba=ab$ leads to a contradiction.
\subsubsection{$G_{1},G_{2},G_{3}\cong\mathbb{Z}_{4}$}
There exists three elements $a,b,c$ in $G\backslash Z(G)$ such that:
\begin{itemize}
\item $G_{1}=\left\{e,a,a^2=\alpha,a^3\right\}$.
\item $G_{2}=\left\{e,b,b^2=\alpha,b^3\right\}$.
\item $G_{3}=\left\{e,c,c^{2}=\alpha,c^3\right\}$.
\end{itemize}
We have $ba(ba)^3=e$ and $baa^3b^3=e$. Consequently $(ba)^3=a^3b^3=ab$; so $(ba)^2ba=ab$; so $\alpha ba=ab$; so $ba=\alpha ab=a^3b$. It is the group of quaternions denoted $\mathbb{H}$. It is not a semi-direct product and its character representation is $a^{4}=b^{4}=e$, $a^{2}=b^{2}$ and $ba=a^{3}b$. It is frequent that $a$ and $b$ are denoted $i$ and $j$, $ab$ is denoted $k$ and $a^2$ is denoted $-1$. With these notation $\mathbb{H}=\left\{1,-1,i,-i,j,-j,k,-k\right\}$ with $i^{2}=j^{2}=-1$, $(-1)^2=1$, $ji=-ij$, $ij=k$.

We have classified all groups of order $8$. The following table gives the non-Abelian groups of order $8$.
\begin{center}
\begin{tabular}{|c|c|}
\hline
\textbf{Name} & \textbf{Character presentation}\\
\hline
$D_{4}\cong\mathbb{Z}_{4}\rtimes\mathbb{Z}_{2}$ & $a^{4}=b^2=e$,$ba=a^3b$\\
\hline
$\mathbb{H}$ & $a^{4}=b^{4}=e$, $a^2=b^2$, $ba=a^3b$\\
\hline
\end{tabular}
\end{center}
\section{Groups of order $16$}
\subsection{Abelian groups of order $16$}
\begin{center}
\begin{tabular}{|c|c|}
\hline
\textbf{Name} & \textbf{Character presentation}\\
\hline
$\mathbb{Z}_{16}$ & $a^{16}=e$\\
\hline
$\mathbb{Z}_{8}\times\mathbb{Z}_{2}$ & $a^{8}=b^{2}=e$, $D(G)=\left\{e\right\}$\\
\hline
$\mathbb{Z}_{4}\times\mathbb{Z}_{4}$ & $a^4=b^4=e$, $D(G)=\left\{e\right\}$\\
\hline
$\mathbb{Z}_{4}\times(\mathbb{Z}_{2})^2$ & $a^4=b^2=c^2=e$, $D(G)=\left\{e\right\}$\\
\hline
$(\mathbb{Z}_{2})^4$ & $a^2=b^2=c^2=d^2=e$, $D(G)=\left\{e\right\}$\\
\hline
\end{tabular}
\end{center}
\subsection{Non-Abelian groups of order 16}
As $\dps{G/Z(G)}$ cannot be cyclic, we show easily that $|Z(G)|\in\left\{2,4\right\}$.
\subsubsection{First case: $|Z(G)|=4$}
In this case, as  $\dps{G/Z(G)}$ is not cyclic, $\dps{G/Z(G)\cong \mathbb{Z}_{2}\times \mathbb{Z}_{2}}$. As $\dps{G/Z(G)}$ is Abelian, then $D(G)\vartriangleleft Z(G)$.\\
As it has been done for the groups of order $8$, we show that $G$ has three sub-groups $G_{1},G_{2},G_{3}$ of order $8$ which contain the center.
\begin{proposition}
The subgroups $G_{1},G_{2},G_{3}$ are Abelian.
\end{proposition}
\begin{proof}
Indeed, as $Z(G)<G_{i}$, we deduce easily that $Z(G)<Z(G_{i})$. Consequently, $Z(G_i)$ has at least $4$ elements. And as the center of a group of order $8$ has either $2$ elements or $8$ elements, then $|Z(G_i)|=8$ and we deduce that $G_i$ is Abelian.
\end{proof}
The following propositions can be proved in a similar way as for the groups of order $8$.
\begin{proposition}
If $g_{1}\in G_{1}\backslash Z(G)$ and $g_{2}\in G_{2}\backslash Z(G)$, then $g_{1}g_{2}\in G_{3}\backslash Z(G)$.
\end{proposition}
\begin{proposition}
If $g_{1}\in G_{1}\backslash Z(G)$ and $g_{2}\in G_{2}\backslash Z(G)$, then $g_{1}g_{2}\neq g_{2}g_{1}$.
\end{proposition}
\begin{proposition}
Let $G_{i},G_{j}$ be two distinct subgroups amongst $G_{1},G_{2},G_{3}$, then:
\begin{eqnarray}
G&=&G_{i}G_{j}\nonumber\\
G_{i}\cap G_{j}&=&Z(G)\nonumber
\end{eqnarray}
\end{proposition}
\paragraph{A) First sub-case: $Z(G)\cong\mathbb{Z}_{2}\times\mathbb{Z}_{2}$}

Let $Z(G)=\left\{e,\alpha,\beta,\gamma\right\}$ with $\alpha^2=\beta^2=\gamma^2=e$.\\
The Abelian groups of order $8$ which have $\mathbb{Z}_{2}\times\mathbb{Z}_{2}$ for subgroup are:
\begin{itemize}
\item $\mathbb{Z}_{4}\times\mathbb{Z}_{2}=\left\{(0,0),(1,0),(2,0),(3,0),(0,1),(1,1),(2,1),(3,1)\right\}$ with:
\begin{itemize}
\item $\left\{(0,0),(2,0),(0,1),(2,1)\right\}\cong\mathbb{Z}_{2}\times\mathbb{Z}_{2}$.
\end{itemize}
\item $(\mathbb{Z}_2)^3=\left\{(0,0,0),(0,0,1),(0,1,0),(0,1,1),(1,0,0),(1,0,1),(1,1,0),(1,1,1)\right\}$ with:
\begin{itemize}
\item $\left\{(0,0,0),(0,0,1),(0,1,0),(0,1,1)\right\}\cong\mathbb{Z}_{2}\times\mathbb{Z}_{2}$.
\item $\left\{(0,0,0),(0,0,1),(1,0,0),(1,0,1)\right\}\cong\mathbb{Z}_{2}\times\mathbb{Z}_{2}$.
\item $\left\{(0,0,0),(0,0,1),(1,1,0),(1,1,1)\right\}\cong\mathbb{Z}_{2}\times\mathbb{Z}_{2}$.
\item $\left\{(0,0,0),(0,1,0),(1,0,0),(1,1,0)\right\}\cong\mathbb{Z}_{2}\times\mathbb{Z}_{2}$.
\item $\left\{(0,0,0),(0,1,0),(1,0,1),(1,1,1)\right\}\cong\mathbb{Z}_{2}\times\mathbb{Z}_{2}$.
\item $\left\{(0,0,0),(0,1,1),(1,0,0),(1,1,1)\right\}\cong\mathbb{Z}_{2}\times\mathbb{Z}_{2}$.
\item $\left\{(0,0,0),(0,1,1),(1,0,1),(1,1,0)\right\}\cong\mathbb{Z}_{2}\times\mathbb{Z}_{2}$.
\end{itemize}
\end{itemize}
\subparagraph{A1) $G_{1},G_{2},G_{3}\cong (\mathbb{Z}_2)^3$:} There exists three elements $a,b,c$ which are not in the center $Z(G)$ such that:
\begin{itemize}
\item $G_{1}=\left\{e,\alpha,\beta,\gamma,a,a\alpha,a\beta,a\gamma\right\}$.
\item $G_{2}=\left\{e,\alpha,\beta,\gamma,b,b\alpha,b\beta,b\gamma\right\}$.
\item $G_{3}=\left\{e,\alpha,\beta,\gamma,c,c\alpha,c\beta,c\gamma\right\}$.
\end{itemize}
$ba$ has order $2$, so $baba=e$. But, $a$ and $b$ have order $2$, so $baab=e$. Consequently $ba=ab$, which leads to a contradiction.
\subparagraph{A2) $G_{1}\cong\mathbb{Z}_{4}\times\mathbb{Z}_{2},G_{2},G_{3}\cong (\mathbb{Z}_2)^3$:}
There exists three elements $a,b,c$ which are not in the center $Z(G)$ such that:
\begin{itemize}
\item $G_{1}=\left\{e,a,a^2=\alpha,a^3,\beta,a\beta,a^2\beta=\gamma,a^3\beta\right\}$.
\item $G_{2}=\left\{e,\alpha,\beta,\gamma,b,b\alpha,b\beta,b\gamma\right\}$.
\item $G_{3}=\left\{e,\alpha,\beta,\gamma,c,c\alpha,c\beta,c\gamma\right\}$.
\end{itemize}
We show easily that $ba=a^3b$ ($ba$ has order $2$, $a$ has order $4$ and $b$ has order $2$). As $G=G_{1}G_{2}$, $G$ is generated by $a,b,\beta$. We have $|<a>|\times |<b>|=8$ and $<a>\cap <b>=\left\{e\right\}$. Consequently $<a,b>\cong D_4$. Moreover, $\beta a=a\beta$ and $\beta b=b\beta$, as $\beta\in Z(G)$. $|<a,b>|\times |<\beta>|=16=|G|$ and $<a,b>\cap<\beta>=\left\{e\right\}$. Consequently, the group $G$ is:
\begin{equation}
G\cong D_4\times\mathbb{Z}_2\nonumber
\end{equation}
Its character presentation is $x^4=y^2=z^2$, $yx=x^3y$, $zx=xz$, $zy=yz$. (In order to avoid confusion, we will prefer character presentation with letter $x,y,z,w,...$ rather than $a,b,c,...$)
\subparagraph{A3) $G_{1},G_{2}\cong\mathbb{Z}_{4}\times\mathbb{Z}_{2},G_{3}\cong (\mathbb{Z}_2)^3$:} We consider two cases:\\
\textbf{First case:}\\ 
There exists three elements $a,b,c$ which are not in the center $Z(G)$ such that:
\begin{itemize}
\item $G_{1}=\left\{e,a,a^2=\alpha,a^3,\beta,a\beta,a^2\beta=\gamma,a^3\beta\right\}$.
\item $G_{2}=\left\{e,b,b^2=\alpha,b^3,\beta,b\beta,b^2\beta=\gamma,b^3\beta\right\}$.
\item $G_{3}=\left\{e,\alpha,\beta,\gamma,c,c\alpha,c\beta,c\gamma\right\}$.
\end{itemize}
Then we show by considering the order of elements that $ba=a^3b^3$. Consequently, $ba=a^3b^3=a^2abb^2=\alpha ab\alpha=\alpha^2ab=ab$, which leads to a contradiction.\\
\textbf{Second case:}\\ 
There exists three elements $a,b,c$ which are not in the center $Z(G)$ such that:
\begin{itemize}
\item $G_{1}=\left\{e,a,a^2=\alpha,a^3,\beta,a\beta,a^2\beta=\gamma,a^3\beta\right\}$.
\item $G_{2}=\left\{e,b,b^2=\beta,b^3,\alpha,b\alpha,b^2\alpha=\gamma,b^3\alpha\right\}=\left\{e,b,b^2=\beta,b^3,\gamma,b\gamma,b^2\gamma=\alpha,b^3\gamma=b\alpha\right\}$.
\item $G_{3}=\left\{e,\alpha,\beta,\gamma,c,c\alpha,c\beta,c\gamma\right\}$.
\end{itemize}
We show that $ba=a^3b^3=\alpha ab\beta=\gamma ab$. Without loss of generality, one can suppose that $ab=c$, so $ba=\gamma c$. We have $b\gamma=\gamma b$ and $bc=bab=\gamma cb$, $<\gamma,c>\cong\mathbb{Z}_2\times\mathbb{Z}_2$, $<b>\cong\mathbb{Z}_4$, $|<\gamma,c>|\times |<b>|=4\times 4=16=|G|$ and $<\gamma,c>\cap <b>=\left\{e\right\}$. Consequently, $G$ is the semi-direct product:
\begin{equation}
G\cong \left(\mathbb{Z}_2\times\mathbb{Z}_2\right)\rtimes_{\varphi}\mathbb{Z}_4,\nonumber
\end{equation}
with:
\begin{eqnarray}
\varphi:\mathbb{Z}_4&\rightarrow&\textrm{Aut}\left(\mathbb{Z}_2\times\mathbb{Z}_2\right)\nonumber\\
1&\rightarrow&\left(\begin{array}{c}
(1,0)\rightarrow(1,0)\\
(0,1)\rightarrow(1,1)
\end{array}\right)\nonumber
\end{eqnarray}
Its minimal character presentation is $x^4=y^4=e$, $yx=x^3y^3$, $x^2,y^2\in Z(G)$. (another possible presentation can be $x^2=y^2=z^4=e$, $yx=xy$, $zx=xyz$, $zy=yz$ but needs more generators).
\subparagraph{A4) $G_{1},G_{2},G_{3}\cong\mathbb{Z}_{4}\times\mathbb{Z}_{2}$:} We consider two cases:\\
\textbf{First case:}\\ 
There exists three elements $a,b,c$ which are not in the center $Z(G)$ such that:
\begin{itemize}
\item $G_{1}=\left\{e,a,a^2=\alpha,a^3,\beta,a\beta,a^2\beta=\gamma,a^3\beta\right\}$.
\item $G_{2}=\left\{e,b,b^2=\beta,b^3,\alpha,b\alpha,b^2\alpha=\gamma,b^3\alpha\right\}$.
\item $G_{3}=\left\{e,c,c^2,c^3,\ldots\right\}$ (the structure of $G_3$ will be precised latter).
\end{itemize}
As $ba$ is not in the center but in $G_3$ then $ba$ has order $4$. In all case, $(ba)^2=c^2$. Indeed, in $\mathbb{Z}_{4}\times\mathbb{Z}_{2}$, the square of all elements of order $4$ is $(2,0)$. We show that $(ba)^2ba=a^3b^3=a^2abb^2=\alpha ab\beta=\gamma ab$. Consequently, $(ba)^2\neq\gamma$. Without loss of generality, one can suppose that $c^2=(ba)^2=\beta$; which gives the entire structure of $G_3$. Consequently, $ba=\beta\gamma ab=\alpha ab=a^2ab=a^3b$. We have $<a>\cong\mathbb{Z}_4$ (idem for $<b>$), $|<a>|\times |<b>|=16=|G|$ and $<a>\cap <b>=\left\{e\right\}$. Consequently, $G$ is the semi-direct product:
\begin{equation}
G\cong\mathbb{Z}_4\rtimes\mathbb{Z}_4\nonumber
\end{equation}
Its character presentation is $x^4=y^4=e$, $yx=x^3y$.\\
\textbf{Second case:}\\
From the first case, one can see that amongst the square $a^2,b^2,c^2$, at least two are equal. The second case is when $a^2=b^2=c^2$. Without loss of generality:
\begin{itemize}
\item $G_{1}=\left\{e,a,a^2=\alpha,a^3,\beta,a\beta,a^2\beta=\gamma,a^3\beta\right\}$.
\item $G_{2}=\left\{e,b,b^2=\alpha,b^3,\beta,b\beta,b^2\beta=\gamma,b^3\beta\right\}$.
\item $G_{3}=\left\{e,c,c^2=\alpha,c^3,\beta,c\beta,c^2\beta=\gamma,c^3\beta\right\}$
\end{itemize}
We show that $(ba)^3=a^3b^3=a^2abb^2=ab$, so $(ba)^2ba=ab$, so $ba=a^3b$. As the character presentation of $<a,b>$ is $a^4=b^4=e$, $a^2=b^2$, $ba=a^3b$, we deduce that $<a,b>\cong\mathbb{H}$. Moreover, $G=G_1G_2=<a,b,\beta>$. As $\beta\in Z(G)$, $|<\beta>|\times |<a,b>|=2\times 8=16$ and $<\beta>\cap<a,b>=\left\{e\right\}$, we deduce that $G$ is the direct product:
\begin{equation}
\mathbb{H}\times\mathbb{Z}_2\nonumber
\end{equation}
Its character presentation is $x^4=y^4=z^2=e$, $x^2=y^2$, $yx=x^3y$, $zx=xz$, $zy=yz$.

We have classified all groups of order $16$ such that $Z(G)\cong\mathbb{Z}_2\times\mathbb{Z}_2$:
\begin{center}
\begin{tabular}{|c|c|}
\hline
\textbf{Name}&\textbf{Character presentation}\\
\hline
$D_{4}\times\mathbb{Z}_{2}$ & $x^4=y^2=z^2=e$,$yx=x^3y$,$zx=xz$,$zy=yz$\\
\hline
$\left(\mathbb{Z}_{2}\times\mathbb{Z}_{2}\right)\rtimes\mathbb{Z}_{4}$& $x^4=y^4=e$, $yx=x^3y^3$, $x^2,y^2\in Z(G)$\\
\hline
$\mathbb{Z}_{4}\rtimes\mathbb{Z}_{4}$ & $x^4=y^4=e$, $yx=x^3y$\\
\hline
$\mathbb{H}\times\mathbb{Z}_{2}$ & $x^4=y^4=z^2=e$, $x^2=y^2$, $yx=x^3y$, $zx=xz$, $zy=yz$\\
\hline
\end{tabular}\\
\textbf{Groups of order 16 with $Z(G)\cong\mathbb{Z}_2\times\mathbb{Z}_{2}$}
\end{center}
\paragraph{B) Second sub-case:} $Z(G)\cong\mathbb{Z}_{4}$.\\
Let $Z(G)=\left\{e,\alpha,\alpha^2,\alpha^3\right\}$.\\
The Abelian groups of order $8$ which have $\mathbb{Z}_{4}$ for subgroups are:
\begin{itemize}
\item $\mathbb{Z}_{8}=\left\{0,1,2,3,4,5,6,7\right\}$ with:
\begin{itemize}
\item $\left\{0,2,4,6\right\}\cong\mathbb{Z}_{4}$
\end{itemize}
\item $\mathbb{Z}_{4}\times\mathbb{Z}_{2}=\left\{(0,0),(1,0),(2,0),(3,0),(0,1),(1,1),(2,1),(3,1)\right\}$ with:
\begin{itemize}
\item $\left\{(0,0),(1,0),(2,0),(3,0)\right\}\cong\mathbb{Z}_{4}$
\item $\left\{(0,0),(1,1),(2,0),(3,1)\right\}\cong\mathbb{Z}_{4}$
\end{itemize}
\end{itemize}
We remark that the square of any element of $\mathbb{Z}_8$ is in the subgroup of order $4$ and the square of any element of $\mathbb{Z}_{4}\times\mathbb{Z}_{2}$ is equal to $(0,0)$ or $(2,0)$ and is in its two subgroups of order $4$. We deduce that the square of any element of $G_{i}$ is in the center. From this fact, we deduce the following proposition.
\begin{proposition}
For $i\neq j$, the commutator $[g_{i},g_{j}]=g_{i}g_{j}g_{i}^{-1}g_{j}^{-1}$ between two elements respectively from $G_i\backslash Z(G)$ and $G_j\backslash Z(G)$ has an order equal to $2$.
\end{proposition}
\begin{proof}
As $g_i^2\in Z(G)$ and using $D(G)<Z(G)$, we deduce $g_jg_i^2=g_i^2g_j=g_ig_ig_j=g_i[g_i,g_j]g_jg_i=[g_i,g_j]g_ig_jg_i=[g_i,g_j]^2g_jg_i^2$ and we deduce the conclusion.
\end{proof}
\subparagraph{B1) $G_1,G_2,G_3\cong\mathbb{Z}_{4}\times\mathbb{Z}_{2}$:} By considering how $\mathbb{Z}_{4}\times\mathbb{Z}_{2}$ is built from one of its subgroup of order $4$ and an element of order $2$ which is not in the considered subgroup, we deduce that there exists $a,b,c$ not in the center and of order $2$ such that:
\begin{itemize}
\item $G_{1}=\left\{e,\alpha,\alpha^2,\alpha^3,a,a\alpha,a\alpha^2,a\alpha^3\right\}$ with $|a\alpha|=4$, $|a\alpha^2|=2$ and $|a\alpha^3|=4$.
\item $G_{2}=\left\{e,\alpha,\alpha^2,\alpha^3,b,b\alpha,b\alpha^2,b\alpha^3\right\}$ with $|b\alpha|=4$, $|b\alpha^2|=2$ and $|b\alpha^3|=4$.
\item $G_{3}=\left\{e,\alpha,\alpha^2,\alpha^3,c,c\alpha,c\alpha^2,c\alpha^3\right\}$ with $|c\alpha|=4$, $|c\alpha^2|=2$ and $|c\alpha^3|=4$.
\end{itemize}
The order of $ba$ is either $2$ or $4$.\\
Suppose that $|ba|=2$. Then $baba=e$ and as $a$ and $b$ have order $2$, then $baab=e$. We deduce $ba=ab$, which leads to a contradiction.\\
Suppose that $|ba|=4$. Then $ba(ba)^3=e$ and $baab=e$. Consequently, $(ba)^3=ab$. From the structure of $\mathbb{Z}_{4}\times\mathbb{Z}_{2}$, we show easily that $(ba)^2=\alpha^2$, so $ba=\alpha^2ab$. As $<\alpha,a>=G_{1}\cong\mathbb{Z}_{4}\times\mathbb{Z}_{2}$, $<b>\cong\mathbb{Z}_{2}$, $|<\alpha,a>|\times |<b>|=8\times 2=16$ and $<\alpha,a>\cap <b>=\left\{e\right\}$; $G$ is the semi-direct product:
\begin{equation}
G\cong\left(\mathbb{Z}_{4}\times\mathbb{Z}_{2}\right)\rtimes_{\varphi}\mathbb{Z}_{2}\nonumber
\end{equation}
with:
\begin{eqnarray}
\varphi:\mathbb{Z}_2&\rightarrow&\textrm{Aut}\left(\mathbb{Z}_4\times\mathbb{Z}_2\right)\nonumber\\
1&\rightarrow&\left(\begin{array}{c}
(1,0)\rightarrow(1,0)\\
(0,1)\rightarrow(2,1)
\end{array}\right)\nonumber
\end{eqnarray}
Its character presentation is $x^4=y^2=z^2=e$, $xy=yx$, $zx=xz$, $zy=x^2yz$.
\subparagraph{B2) $G_1\cong\mathbb{Z}_{8},G_2,G_3\cong\mathbb{Z}_{4}\times\mathbb{Z}_{2}$:}
There exists an element $a$ not in the center and of order $8$ and two elements $b,c$ not in the center of order $2$ such that:
\begin{itemize}
\item $G_{1}=\left\{e,a,a^2=\alpha,a^3,a^4=\alpha^2,\ldots\right\}$.
\item $G_{2}=\left\{e,\alpha,\alpha^2,\alpha^3,b,b\alpha,b\alpha^2,b\alpha^3\right\}$ with $|b\alpha|=4$, $|b\alpha^2|=2$ and $|b\alpha^3|=4$.
\item $G_{3}=\left\{e,\alpha,\alpha^2,\alpha^3,c,c\alpha,c\alpha^2,c\alpha^3\right\}$ with $|c\alpha|=4$, $|c\alpha^2|=2$ and $|c\alpha^3|=4$.
\end{itemize}
Let us consider that $|ba|=2$, then $baba=e$ and $baa^7b=e$. Consequently, $ba=a^7b=a^6ab=\alpha^3ab$. But $\alpha^3$, which is the commutator $[b,a]$ is of order $4$; which is contradictory.\\
Consider $|ba|=4$, then we show that $(ba)^2ba=a^7b=\alpha^3ab$. As $(ba)^2=\alpha^2$; then $ba=\alpha ab$; which is also contradictory.
\subparagraph{B3) $G_1,G_2\cong\mathbb{Z}_{8},G_3\cong\mathbb{Z}_{4}\times\mathbb{Z}_{2}$:}
There exists two elements $a,b$ which are not in the center and of order $8$ and an element $c$ not in the center of order $2$ such that:
\begin{itemize}
\item $G_{1}=\left\{e,a,a^2=\alpha,a^3,a^4=\alpha^2,\ldots\right\}$.
\item $G_{2}=\left\{e,b,b^2=\alpha,b^3,b^4=\alpha^2,\ldots\right\}$.
\item $G_{3}=\left\{e,\alpha,\alpha^2,\alpha^3,c,c\alpha,c\alpha^2,c\alpha^3\right\}$ with $|c\alpha|=4$, $|c\alpha^2|=2$ and $|c\alpha^3|=4$.
\end{itemize}
If $|ba|=2$, then $ba=a^7b^7=a^6abb^6=\alpha^3ab\alpha^3=\alpha^2ab$. We have also $|ab|=2$. Without loss of generality, one can suppose $ab=c$. We deduce $ca=aba=a\alpha^2ab=a\alpha^2c=a^5c$. $<a>\cong\mathbb{Z}_{8}$, $<c>\cong\mathbb{Z}_{2}$, $|<a>|\times |<c>|=8\times2=16$ and $<a>\cap <c>=\left\{e\right\}$. Consequently, $G$ is the semi-direct product:
\begin{equation}
G\cong\mathbb{Z}_{8}\rtimes_{\varphi_1}\mathbb{Z}_{2},\nonumber
\end{equation}
where:
\begin{eqnarray}
\varphi_{1}:\mathbb{Z}_{2}&\rightarrow&\textrm{Aut}\left(\mathbb{Z}_{8}\right)\nonumber\\
1&\rightarrow&(\begin{array}{c}
1\rightarrow 5
\end{array})\nonumber
\end{eqnarray}
Its character presentation is $x^8=y^2=e$, $yx=x^5y$.\\
If $|ba|=4$, then $(ba)^3=\alpha^2ab$. As $(ba)^2=\alpha^2$, then $ba=ab$, which is contradictory.
\subparagraph{B4) $G_1,G_2,G_3\cong\mathbb{Z}_{8}$:}
Then, there exists $a,b$ not in the center of order $8$ and such that $ba$ is also of order $8$. We deduce $(ba)^7=a^7b^7=\alpha^3ab\alpha^3=\alpha^2ab$. As $(ba)^6=\alpha^2$ then $ba=ab$; which leads to a contradiction.

We have classified all groups of order $16$ such that $Z(G)\cong\mathbb{Z}_4$:
\begin{center}
\begin{tabular}{|c|c|}
\hline
\textbf{Name}&\textbf{Character presentation}\\
\hline
$\left(\mathbb{Z}_4\times\mathbb{Z}_2\right)\rtimes\mathbb{Z}_2$&$x^4=y^2=z^2=e$,$xy=yx$,$zx=xz$,$zy=x^2yz$\\
\hline
$\mathbb{Z}_8\rtimes_{\varphi_1}\mathbb{Z}_2$&$x^8=y^2=e$,$yx=x^5y$\\
\hline
\end{tabular}\\
\textbf{Groups of order 16 with $Z(G)\cong\mathbb{Z}_4$}
\end{center}
\subsubsection{Second case: $|Z(G)|=2$}
This case is more difficult as $G/Z(G)$ is not necessarily Abelian.\\
In the following, we will denote $Z(G)=\left\{e,z\right\}$.\\
We have to study the different cases for $G/Z(G)$. $G/Z(G)$ is not cyclic. The different cases are presented below.
\paragraph{A) $G/Z(G)\cong\mathbb{Z}_4\times\mathbb{Z}_2$:} $ $\\The maximal subgroups of $\mathbb{Z}_4\times\mathbb{Z}_2=\left\{(0,0),(1,0),(2,0),(3,0),(0,1),(1,1),(2,1),(3,1)\right\}$ are:
\begin{itemize}
\item $\left\{(0,0),(1,0),(2,0),(3,0)\right\}\cong \mathbb{Z}_4$ and has only one subgroup of order $2$.
\item $\left\{(0,0),(1,1),(2,0),(3,1)\right\}\cong \mathbb{Z}_4$ and has only one subgroup of order $2$.
\item $\left\{(0,0),(0,1),(2,0),(2,1)\right\}\cong \mathbb{Z}_2\times \mathbb{Z}_2$ and has three subgroups of order $2$.
\end{itemize}
Consequently, $G$ has $3$ subgroups $G_{1},G_{2},G_{3}$ of order $8$ containing the center such that $G_{1},G_{2}$ has only one subgroup of order $4$ containing the center and $G_{3}$ has three subgroups of order $4$ containing the center. We show that $G_{1}$ (resp. $G_{2}$) is Abelian. Indeed, $Z(G)<Z(G_{1})$. If $G_{1}$ were not Abelian, then it has three subgroups of order $4$ containing $Z(G_1)$ and consequently containing $Z(G)$; which is contradictory.\\
The following proposition conducts to a contradiction:
\begin{proposition}
If $G$ has at least two Abelian subgroups of order $8$ then $|Z(G)|\geq 4$.
\end{proposition}
\begin{proof}
Using:
\begin{equation}
|G_1G_2|=\frac{|G_1||G_2|}{|G_1\cap G_2|}=\frac{64}{|G_1\cap G_2|},\nonumber
\end{equation}
we deduce that $G=G_1G_2$ and $|G_1\cap G_2|=4$.\\
If $G_1$ and $G_2$ are Abelian, then taking $g\in G_1\cap G_2$ and $g_{1}g_{2}\in G$. $gg_{1}g_{2}=g_{1}gg_{2}=g_{1}g_{2}g$. Consequently, $G_1\cap G_2\subset Z(G)$. We deduce the result.
\end{proof}
\paragraph{B) $G/Z(G)\cong(\mathbb{Z}_2)^3$:}
In this case, for any $g\in G$, $g^2\in Z(G)$. As $G$ is not Abelian, then there exists an element $g$ such that $g^2=z$, so $|g|=4$. The group $G$ has no element of order $8$ nor $16$. Let $<g>=\left\{e,g,g^2=z,g^3\right\}\cong\mathbb{Z}_4$. It is a subgroup of $G$ containing $Z(G)$. Consequently, by the correspondance theorem, $<g>/Z(G)$ is a subgroup of order $2$ of $G/Z(G)$. Without loss of generality, one can suppose that this subgroup is identified to the subgroup $\left\{(0,0,0),(0,0,1)\right\}$ of $(\mathbb{Z}_2)^3$. This subgroup is contained to the following three subgroups of order 4:
\begin{itemize}
\item $\left\{(0,0,0),(0,0,1),(0,1,0),(0,1,1)\right\}$.
\item $\left\{(0,0,0),(0,0,1),(1,0,0),(1,0,1)\right\}$.
\item $\left\{(0,0,0),(0,0,1),(1,1,0),(1,1,1)\right\}$.
\end{itemize}
Consequently, $G$ has three subgroups $G_{1},G_{2},G_{3}$ of order $8$ containing $<g>$. One can easily show, using the correspondance theorem, that for $i\neq j$ and for $g_{i}\in G_{i}\backslash <g>$, $g_{j}\in G_{j}\backslash <g>$ then $g_{i}g_{j}\in G_{k}\backslash <g> $ for $k\notin\left\{i,j\right\}$. From the proposition 14, the maximum number of Abelian subgroups of order $8$ is $1$. Consequently, one can suppose that $G_{2}$ and $G_{3}$ are non-Abelian. Consequently, one can take $g_{2}\in G_{2}\backslash <g>$ (resp. $g_{3}\in G_{3}\backslash <g>$) which does not commute with $g$. As $G/Z(G)$ is Abelian, then $D(G)=Z(G)$. Consequently, $gg_{2}=zg_{2}g$ and $gg_{3}=zg_{3}g$. We deduce $gg_{2}g_{3}=zg_{2}gg_{3}=zg_{2}zg_{3}g=z^2g_{2}g_{3}g=g_{2}g_{3}g$. Consequently, $g$ commutes with all elements of $G_{1}\backslash<g>$; and as a consequence, $g\in Z(G_{1})$. $Z(G_{1})$ contains $\left\{e,z,g\right\}$. We deduce that $G_1$ is Abelian. As $G$ does not contains element of order $8$ and as $g\in G_1$ is of order $4$ then $G_1\cong\mathbb{Z}_{4}\times\mathbb{Z}_{2}$. Let us denote:\\
$G_1=\left\{e,g,g^2=z,g^3,g_1,gg_1,g^2g_1,g^3g_1\right\}$. Let $g_2\in G_2\backslash <g>$ (corresponds to a symmetry in $D_4$ or one of the four other elements of order $4$ in $\mathbb{H}$). If one of the $4$ elements of $G_1\backslash <g>$ commutes with one of the elements of $G_2\backslash <g>$, then it commutes with all elements of $G_{2}$ and the center of $G_2$ has at least $3$ elements, which is contradictory. Consequently; none element of $G_1\backslash<g>$ commutes with one element of $G_2\backslash<g>$.\\
Let $g_{1}\in G_1\backslash<g>$ and $g_{2}\in G_2\backslash<g>$, then $g_1g_2=zg_2g_1$. We show that $gg_1g_2=gzg_2g_1=zgg_2g_1=z^2g_2gg_1=g_2gg_1$. Consequently, $g_2$ commutes with $gg_1$; it leads to a contradiction.
\paragraph{C) $G/Z(G)\cong\mathbb{H}$:} $ $\\
The maximal subgroups of $\mathbb{H}=\left\{1,-1,i,-i,j,-j,k,-k\right\}$ are:
\begin{itemize}
\item $\left\{1,i,-1,-i\right\}\cong\mathbb{Z}_4$, which has only one subgroup of order $2$.
\item $\left\{1,j,-1,-j\right\}\cong\mathbb{Z}_4$, which has only one subgroup of order $2$.
\item $\left\{1,k,-1,-k\right\}\cong\mathbb{Z}_4$, which has only one subgroup of order $2$.
\end{itemize}
Consequently, $G$ has three Abelian subgroups of order $8$. Using the proposition 14, it leads to a contradiction.
\paragraph{D) $G/Z(G)\cong D_4$:} $ $\\
The maximal subgroups of $D_4=\left\{Id,r,r^2,r^3,s,rs,r^2s,r^3s\right\}$ are:
\begin{itemize}
\item $\left\{Id,r,r^2,r^3\right\}\cong \mathbb{Z}_4$ which has only one subgroup of order $2$.
\item $\left\{Id,s,r^2,r^2s\right\}\cong \mathbb{Z}_2\times \mathbb{Z}_2$ which has three subgroups of order $2$.
\item $\left\{Id,rs,r^2,r^3s\right\}\cong \mathbb{Z}_2\times \mathbb{Z}_2$ which has three subgroups of order $2$.
\end{itemize}
Consequently, $G$ has three subgroups of order $8$, $G_1,G_2,G_3$ which contain the center such that $G_1$ has only one subgroup of order $4$ which contains the center and $G_2,G_3$ have three subgroups of order $4$ which contain the center. $G_{1}$ is Abelian and $G_{2},G_{3}$ are necessarily non-Abelian. We show easily $Z(G)=Z(G_2)=Z(G_3)$ as the center of a non-Abelian group of order $8$ has two elements. If $G_2\cong D_4$, as $r^2=rsr^3s$, then $r^2\in D(G)$. If $G_2\cong\mathbb{H}$, as $-1=ij(-i)(-j)$, then $-1\in D(G)$. In all cases, $Z(G)<D(G)$ and the inclusion is strict as $G/Z(G)$ is non-Abelian. As, in $D_4$, $s$ and $rs$ do not commute; then there exists $g_{2}\in G_{2}$ and $g_{3}\in G_3$ such that $g_2Z(G)$ (identified to $s$) and $g_3Z(G)$ (identified to $rs$) do not commute. Consequently, there exists $g'\in G$ (such $g'Z(G)$ is identified to $r^2$) such $g_2g_3Z(G)=g_3g_2g'Z(G)$. We deduce that there exists $z_0\in Z(G)$ such $g_2g_3=g_3g_2g'z_0$. Consequently, $g'z_0\in D(G)$ and as $z_0\in D(G)$, we deduce that $g'\in D(G)$.\\
We have $\left\{e,g',z,g'z\right\}\subset D(G)$ and forms a group. Let us show that this group is a normal subgroup of $G$. We have $hg'zh^{-1}Z(G)=hg'h^{-1}Z(G)=g'Z(G)$ as $g'Z(G)\in Z\left(G/Z(G)\right)$. Consequently $hg'h^{-1},hg'zh^{-1}\in\left\{g',g'z\right\}$; we deduce that $\left\{e,g',z,g'z\right\}$ is a normal subgroup of $G$. As $G/\left\{e,g',z,g'z\right\}$ is of order $4$, it is necessarily Abelian.\\
Consequently, $D(G)=\left\{e,g',z,g'z\right\}$. Moreover, $\dps{D(G)/Z(G)=Z\left(G/Z(G)\right)}$ and corresponds to $\left\{Id,r^2\right\}$. We deduce that $D(G)<G_i$ for any $i=1,2,3$. Using the formula $\dps{|G_iG_j=\frac{|G_i||G_j|}{|G_i\cap G_j|}|}$, we deduce $G=G_iG_j$ and $G_i\cap G_j=D(G)$.
\subparagraph{D1) $D(G)\cong\mathbb{Z}_2\times\mathbb{Z}_2$:}$ $\\
$G_1$ cannot be isomorphic to $(\mathbb{Z}_2)^3$, as any subgroup of order $2$ is included in three subgroups of order $4$. The only possibility is $G_1\cong\mathbb{Z}_4\times\mathbb{Z}_2$. The unique subgroup of order $4$ containing the center is then identified to $\left\{(0,0),(2,0),(0,1),(2,1)\right\}$. As the square of any element of $\mathbb{Z}_4\times\mathbb{Z}_2$ is either $(0,0)$ or $(2,0)$ and $(2,0)$ is in three subgroups of order $4$; we deduce that there exists an element $a$ of order $4$ in $G\backslash D(G)$ such that:
\begin{itemize}
\item $G_1=\left\{e,a,a^2=g',a^3,z,az,a^2z=g'z,a^3z\right\}\cong\mathbb{Z}_4\times\mathbb{Z}_2$.
\end{itemize}
As $\mathbb{H}$ has no subgroup isomorphic to $\mathbb{Z}_2\times\mathbb{Z}_2$; we deduce $G_{2},G_{3}\cong D_4$. As $Z(G)=Z(G_2)=Z(G_3)$, then there exists $b,c$ in $G\backslash D(G)$ of order $4$ such that:
\begin{itemize}
\item $G_2=\left\{e,b,b^2=z,b^3,g',bg'=g'b^3,b^2g'=zg',b^3g'=g'b\right\}$.
\item $G_3=\left\{e,c,c^2=z,c^3,g',cg'=g'c^3,c^2g'=zg',c^3g'=g'c\right\}$.
\end{itemize}
By correspondance theorem, we show that $baZ(G)=abg'Z(G)=g'abZ(G)$.\\
Suppose that $ba=g'ab$. If $|ba|=2$, then $ba=a^3b^3=a^2abb^2=g'abz=g'zab$. Consequently, $g'=g'z$, which is contradictory. We have $ab=g'ba$. If we suppose that $|ba|=4$, as $g'$ is a symmetry, then $|ab|=2$ and it conducts to the same contradiction.\\
Suppose that $ba=g'zab$. If $|ba|=4$, then $(ba)^2ba=a^3b^3=g'abz=g'zab$. As $(ba)^2=z$ then $ba=g'ab$ which is contradictory. We have $ab=g'zba$. If $|ba|=2$, as $ba$ is not in $D(G)$, we deduce $|ab|=4$ which leads to the same contradiction.\\
Consequently, it is impossible that $D(G)\cong\mathbb{Z}_2\times\mathbb{Z}_2$.
\subparagraph{D2) $D(G)\cong\mathbb{Z}_4$:}$ $\\
In this case, $D(G)=\left\{e,g',g'^2=z,g'^3\right\}$
We have seen that $G_{1}$ cannot be isomorphic to $\left(\mathbb{Z}_2\right)^3$. However, as $\mathbb{Z}_{4}\times\mathbb{Z}_{2}$ has two groups isomorphic to $\mathbb{Z}_4$ which contains a common subgroup of order $2$; $G_1$ cannot be isomorphic to $\mathbb{Z}_{4}\times\mathbb{Z}_{2}$ (as the center will be included in two subgroups of order $4$). Consequently, $G_1\cong\mathbb{Z}_8$ and there exists $a\in G\backslash D(G)$ of order $8$ such that:
\begin{itemize}
\item $G_1=\left\{e,a,a^2=g',a^3,a^4=z,\ldots\right\}$.
\end{itemize}
\subparagraph{D2-1) $G_2,G_3\cong D_4$:}$ $\\
There exists $b,c$ not in $D(G)$ and of order $2$ such that:
\begin{itemize}
\item $G_{2}=\left\{e,g',g'^2=z,g'^3,b,g'b=bg'^3,g'^2b,g'^3b=bg'\right\}$.
\item $G_{3}=\left\{e,g',g'^2=z,g'^3,c,g'c=cg'^3,g'^2c,g'^3c=cg'\right\}$.
\end{itemize}
By the correspondance theorem and the structure of the maximal subgroups of $D_4$, we deduce that $ba\in G^3\backslash D(G)$. Consequently, $|ba|=2$ and $ba=a^7b$. $<a>=G_1\cong\mathbb{Z}_8$, $<b>\cong\mathbb{Z}_2$, $|<a>|\times|<b>|=8\times 2=16$ and $<a>\cap<b>=\left\{e\right\}$. Consequently, $G$ is the semi-direct product, also called diedral group $D_8$:
\begin{equation}
G\cong\mathbb{Z}_8\rtimes_{\varphi_2}\mathbb{Z}_2,\nonumber
\end{equation}
with:
\begin{eqnarray}
\varphi_2:\mathbb{Z}_2&\rightarrow&\textrm{Aut}\left(\mathbb{Z}_8\right)\nonumber\\
1&\rightarrow&(1\rightarrow 7)\nonumber
\end{eqnarray}
Its character presentation is $x^8=y^2=e$, $yx=x^7y$.
\subparagraph{D2-2) $G_2\cong D_4,G_3\cong\mathbb{H}$:}$ $\\
There exists $b,c$ not in $D(G)$ respectively of order $2$ and $4$ such that:
\begin{itemize}
\item $G_{2}=\left\{e,g',g'^2=z,g'^3,b,g'b=bg'^3,g'^2b,g'^3b=bg'\right\}$.
\item $G_{3}=\left\{e,g',g'^2=z,g'^3,c,g'c,g'^2c,g'^3c\right\}$.
\end{itemize}
$ba\in G_{3}\backslash D(G)$, then $|ba|=4$. Consequently, $(ba)^2ba=a^7b$. As $(ba)^2=g'^2=z=a^4$, then $ba=a^3b$. With the same method as previously, we deduce that $G$ is the semi-direct product (also called semi-diedral group $SD_8$):
\begin{equation}
G\cong\mathbb{Z}_8\rtimes_{\varphi_3}\mathbb{Z}_2,\nonumber
\end{equation}
with:
\begin{eqnarray}
\varphi_3:\mathbb{Z}_2&\rightarrow&\textrm{Aut}\left(\mathbb{Z}_8\right)\nonumber\\
1&\rightarrow&(1\rightarrow 3)\nonumber
\end{eqnarray}
Its character presentation is $x^8=y^2=e$, $yx=x^3y$.
\subparagraph{D2-3) $G_2,G_3\cong\mathbb{H}$:}$ $\\
There exists $b$ and $c$ not in $D(G)$ of order $4$ such that:
\begin{itemize}
\item $G_{2}=\left\{e,g',g'^2=z,g'^3,b,g'b,g'^2b,g'^3b\right\}$.
\item $G_{3}=\left\{e,g',g'^2=z,g'^3,c,g'c,g'^2c,g'^3c\right\}$.
\end{itemize}
$|ba|=4$ then $(ba)^2ba=a^7b^3$. As $(ba)^2=z=a^4$ then $ba=a^3b^3$. It is not a semi-direct product. Its character presentation is $x^8=y^4=e$, $x^4=y^2$, $yx=x^3y^3$.

The following Table gives the groups of order $16$ such that $Z(G)\cong \mathbb{Z}_2$:
\begin{center}
\begin{tabular}{|c|c|}
\hline
\textbf{Name} & \textbf{Character presentation}\\
\hline
$\mathbb{Z}_{8}\rtimes_{\varphi_2} \mathbb{Z}_{2}$ & $x^8=y^2=e$,$yx=x^7y$\\
\hline
$\mathbb{Z}_{8}\rtimes_{\varphi_3} \mathbb{Z}_{2}$ & $x^8=y^2=e$,$yx=x^3y$\\
\hline
Generalized Quaternions & $x^8=y^4=e$, $x^4=y^2$, $yx=x^3y^3$\\
\hline
\end{tabular}
\end{center}
\section{Conclusion}
The groups of order 16 are given in the following table:
\begin{center}
\begin{tabular}{|c|c|c|}
\hline
\textbf{Name} & \textbf{Character presentation}&\textbf{Center}\\
\hline
$\mathbb{Z}_{16}$ & $a^{16}=e$ & $Z(G)=G$\\
\hline
$\mathbb{Z}_{8}\times\mathbb{Z}_{2}$ & $a^{8}=b^{2}=e$,$ba=ab$ & $Z(G)=G$\\
\hline
$\mathbb{Z}_{4}\times\mathbb{Z}_{4}$ & $a^{4}=b^{4}=e$,$ba=ab$ & $Z(G)=G$\\
\hline
$\mathbb{Z}_{4}\times\left(\mathbb{Z}_{2}\right)^2$ & $a^4=b^2=c^2=e$,$D(G)=\left\{e\right\}$ & $Z(G)=G$\\
\hline
$\left(\mathbb{Z}_{2}\right)^4$ & $a^2=b^2=c^2=d^2=e$,$D(G)=\left\{e\right\}$ & $Z(G)=G$\\
\hline
\hline
$D_{4}\times\mathbb{Z}_{2}$ & $x^4=y^2=z^2=e$,$yx=x^3y$,$zx=xz$,$zy=yz$& $\left\{e,x^2,z,zx^2\right\}\cong\mathbb{Z}_{2}\times\mathbb{Z}_{2}$\\
\hline
$\left(\mathbb{Z}_{2}\times\mathbb{Z}_{2}\right)\rtimes\mathbb{Z}_{4}$ & $x^4=y^4=e$,$yx=x^3y^3$&$\left\{e,x^2,y^2,x^2y^2\right\}\cong\mathbb{Z}_{2}\times\mathbb{Z}_{2}$\\
\hline
$\mathbb{Z}_{4}\rtimes\mathbb{Z}_{4}$ & $x^4=y^4=e$,$yx=x^3y$ & $\left\{e,x^2,y^2,x^2y^2\right\}\cong\mathbb{Z}_{2}\times\mathbb{Z}_{2}$\\
\hline
$\mathbb{H}\times\mathbb{Z}_{2}$ & $x^4=y^4=z^2=e$,$x^2=y^2$,$yx=x^3y$,$zx=xz$,$zy=yz$ & $\left\{e,x^2,z,x^2z\right\}\cong\mathbb{Z}_{2}\times\mathbb{Z}_{2}$\\
\hline
\hline
$\left(\mathbb{Z}_{4}\times\mathbb{Z}_{2}\right)\rtimes\mathbb{Z}_{2}$ & $x^4=y^2=z^2=e$,$xy=yx$,$zx=xz$,$zy=x^2yz$& $\left\{e,x,x^2,x^3\right\}\cong\mathbb{Z}_4$\\
\hline
$\mathbb{Z}_{8}\rtimes_{\varphi_{1}}\mathbb{Z}_{2}$& $x^8=y^2=e$,$yx=x^5y$&$\left\{e,x^2,x^4,x^6\right\}\cong\mathbb{Z}_4$\\
\hline
\hline
$\mathbb{Z}_{8}\rtimes_{\varphi_{2}}\mathbb{Z}_{2}$& $x^8=y^2=e$,$yx=x^7y$&$\left\{e,x^4\right\}\cong\mathbb{Z}_{2}$\\
\hline
$\mathbb{Z}_{8}\rtimes_{\varphi_{3}}\mathbb{Z}_{2}$& $x^8=y^2=e$,$yx=x^3y$&$\left\{e,x^4\right\}\cong\mathbb{Z}_{2}$\\
\hline
Gen. Quat. & $x^8=y^4=e$,$x^4=y^2$,$yx=x^3y^3$&$\left\{e,x^4\right\}\cong\mathbb{Z}_{2}$\\
\hline
\end{tabular}
\end{center}

\end{document}